\documentclass{article}
\usepackage [utf8] {inputenc}
\usepackage{mathtools}
\usepackage{amsthm}
\usepackage{amssymb}

\newtheorem{theorem}{Theorem}
\newtheorem{lemma}[theorem]{Lemma}
\newtheorem{proposition}[theorem]{Proposition}
\newtheorem{corollary}[theorem]{Corollary}

\title{Group topologies on integers and S-unit equations}
\author{S.V.~Skresanov \thanks{The work was supported by the program of fundamental scientific researches
of the SB RAS No. I.1.1., Project No. 0314-2019-0001.}}

\date{}

\begin{document}
\maketitle

\begin{abstract}
	A sequence of integers \( \{ s_n \}_{n \in \mathbb{N}} \) is called a
	T-sequence if there exists a Hausdorff group topology on \( \mathbb{Z} \) such that
	\( \{ s_n \}_{n \in \mathbb{N}} \) converges to zero. For every finite set of primes \( S \)
	we build a Hausdorff group topology on \( \mathbb{Z} \)	such that every growing sequence of \( S \)-integers
	converges to zero. As a corollary, we solve in the affirmative an open problem
	by I.V.~Protasov and E.G.~Zelenuk asking if \( \{ 2^n + 3^n \}_{n \in \mathbb{N}} \) is a T-sequence.
	Our results rely on a nontrivial number-theoretic fact about \( S \)-unit equations.
	\medskip

\end{abstract}

\section{Introduction}

In \cite{prot} I.V.~Protasov and E.G.~Zelenuk introduced the concept of a \textit{T-sequence}:
a sequence of integers \( \{ s_n \}_{n \in \mathbb{N}} \) is called a T-sequence
if there exists a Hausdorff group topology on \( \mathbb{Z} \) such that \( \{ s_n \}_{n \in \mathbb{N}} \)
converges to \( 0 \).
This new tool proved to be useful in constructing topological groups
with required properties and also showed some profound connections to number theory.
For instance, if \( \lim_{n \to \infty} \frac{s_{n+1}}{s_n} = \infty \) or
\( \lim_{n \to \infty} \frac{s_{n+1}}{s_n} = \alpha \), where \( \alpha \) is transcendental,
then \( \{ s_n \}_{n \in \mathbb{N}} \) is a T-sequence (see \cite[Theorems~10 and~11]{prot}).
Certain recurrences define T-sequences, for example, Fibonacci sequence is a T-sequence (see \cite[Theorem~14 and its corollary]{prot}).
Also, every T-sequence has natural density \( 0 \) \cite[Lemma~6]{prot}.

In this short note we address a question posed by I.V.~Protasov and E.G.~Zelenuk in \cite[Question~1]{openpr}
(cf. \cite[Question~15.79]{kourovka}): % also
\medskip

\noindent
\textit{Is \( \{ 2^n + 3^n \}_{n \in \mathbb{N}} \) a T-sequence?}
\medskip

This sequence fails the sufficient conditions that we have mentioned before and has natural density \( 0 \).
Notice also that \( \{ 2^n \}_{n \in \mathbb{N}} \) and \( \{ 3^n \}_{n \in \mathbb{N}} \) are T-sequences,
since they converge to \( 0 \) in a \( 2 \)-adic and \( 3 \)-adic topologies respectively, but clearly no \( p \)-adic
topology will work for \( \{ 2^n + 3^n \}_{n \in \mathbb{N}} \).

Despite those difficulties, we answer that question in the affirmative. In fact, we are able to prove that a
class of T-sequences is sufficiently wide. For instance, a sum of several growing geometric progressions
will always constitute a T-sequence (see corollary below), so in particular, \( \{ 2^n + 3^n \}_{n \in \mathbb{N}} \)
is a T-sequence.

We introduce several definitions. Let \( S \) be a finite set of primes.
An integer is called an \textit{\( S \)-integer} if it is divisible only by primes from \( S \).
For two topologies \( \tau \) and \( \sigma \) we say that \( \sigma \) is \textit{stronger} than \( \tau \)
if \( \tau \subseteq \sigma \). If \( \{ s_n \}_{n \in \mathbb{N}} \) is a T-sequence, then there exists the
strongest Hausdorff group topology on \( \mathbb{Z} \) where \( \{ s_n \}_{n \in \mathbb{N}} \) converges
to zero. It is unique (see~\cite{prot}) so let \( \mathbb{Z}\{s_n\} \) denote the group \( \mathbb{Z} \)
endowed with that topology.

\begin{theorem}\label{mainth}
	Let \( S \) be a nonempty finite set of primes. There exists a T-sequence of
	\( S \)-integers \( \{ s_n \}_{n \in \mathbb{N}} \)
	such that for every sequence of \( S \)-integers \( \{ d_n \}_{n \in \mathbb{N}} \)
	with \( \lim_{n \to \infty} |d_n| = \infty \),
	\( \{ d_n \}_{n \in \mathbb{N}} \) converges to \( 0 \) in \( \mathbb{Z}\{ s_n \} \).
\end{theorem}

We mention that the topology on \( \mathbb{Z}\{ s_n \} \) has several interesting properties,
for example, it is complete and not metrizable (see \cite[Theorem~8 and Corollary to Theorem~6]{prot}).

Our theorem deals with sequences where each term has prescribed prime factors
(they must lie in a given finite set). Observe that it is impossible to generalize this result to sequences
where each term has a uniformly bounded number of distinct prime factors.
Indeed, the sequence of primes \( \{ p_n \}_{n \in \mathbb{N}} \) is not a T-sequence
as there are infinitely many solutions to the equation \( {p_{n+1} - p_n = M} \), where \( M \leq 600 \)
(see \cite{gaps} and Lemma~\ref{newcrit} below).

The positive answer to the above question by Protasov and Zelenuk
is provided by the following application of Theorem~\ref{mainth}.
\begin{corollary}\label{geom}
	Let \( d_i, a_i \), \( i = 1, \dots, k \), be arbitrary integers. If
	\( a_i \neq \pm 1 \) for all \( i = 1, \dots, k \), then 
	\( \{ d_1a_1^n + \dots + d_ka_k^n \}_{n \in \mathbb{N}} \) is a T-sequence.
\end{corollary}
\begin{proof}
	We may assume that all \( a_i \), \( i = 1, \dots, k \), are nonzero.
	Let \( S \) be the set of all prime divisors of \( a_1 \dots a_k \).
	Clearly \( a_i^n \), \( i = 1, \dots, k \), are \( S \)-integers for all~\( n \),
	so by the main theorem, there exists a Hausdorff group topology on \( \mathbb{Z} \)
	where each sequence \( \{ a_i^n \}_{n \in \mathbb{N}} \) converges to zero.
	Every integer linear combination of sequences converging to zero also
	converges to zero, hence we are done.
\end{proof}

Notice that the requirement \( a_i \neq \pm 1 \) is essential, since, for instance, the sequence \( \{ 1^n \}_{n \in \mathbb{N}} \)
cannot converge to \( 0 \) in a Hausdorff group topology.

The proof of the main result makes use of the so-called \textit{unit equations},
which are essentially Diophantine equations. We mention that previously
unit equations were applied to T-sequences by M.~Higasikawa in~\cite{masasi},
where it was shown that certain properties of topological groups are not preserved
under taking products.

\section{Proof of the main result}

We start with several definitions.
Let \( S \) be a finite set of primes. A rational number is called an \textit{\( S \)-unit}
if it belongs to the multiplicative group generated by \( S \cup \{ -1 \} \). The set
of \( S \)-units is denoted by \( U_S \). Notice that \( S \)-integers are \( S \)-units.

Let \( c_i \), \( i = 1, \dots, k \), and \( M \) be nonzero integers.
We say that a solution \( (x_1, \dots, x_k) \) to the equation
\begin{equation}\label{seq}
	c_1x_1 + \dots + c_kx_k = M
\end{equation}
is \textit{non-degenerate} if \( \sum_{i \in I} c_ix_i \neq 0 \)
for all nonempty subsets \( I \) of \( \{ 1, \dots, k \} \).

\begin{proposition}\textup{\cite[Corollary~6.1.2]{sunits}}\label{finsols}
	Let \( S \) be a finite set of primes.
	Equation~\textup{(\ref{seq})} has finitely many non-degenerate solutions in \( S \)-units \( x_1, \dots, x_k \in U_S \).
\end{proposition}

It is easy to see that there can be infinitely many degenerate solutions to such an equation.
For example, \( x_1 + x_2 + x_3 = 1 \) has infinitely many solutions of the form \( (x, -x, 1) \).
Nevertheless, a certain finiteness condition holds for arbitrary solutions of~(\ref{seq}).

\begin{lemma}\label{inhom}
	Let \( S \) be a finite set of primes.
	Given equation~\textup{(\ref{seq})}, there exist finite sets \( V_i \), \( i = 1, \dots, k \),
	such that for any solution \( (x_1, \dots, x_k) \in U_S^k \) of the equation,
	there exists an index \( i \in \{ 1, \dots, k \} \) with \( x_i \in V_i \).
\end{lemma}
\begin{proof}
	For a nonempty \( I \subseteq \{ 1, \dots, k \} \) and \( i \in I \), define
	\[ V(I, i) = \{ x_i \mid \{ x_j \}_{j \in I} \subseteq U_S
	   \text{ is a non-degenerate solution to } \sum_{j \in I} c_j x_j = M \}. \]
	By Proposition~\ref{finsols}, there are finitely many non-degenerate solutions in \( S \)-units
	to the equation \( \sum_{j \in I} c_j x_j = M \), so the sets \( V(I, i) \) are finite for all
	\( I \subseteq \{ 1, \dots, k \} \) and \( i \in I \).

	Now define \( V_i = \bigcup_{i \in I \subseteq \{ 1, \dots, n \}} V(I, i) \). Obviously all
	sets \( V_i \) are finite. Suppose that \( (x_1, \dots, x_k) \in U_S^k \) is a solution
	to~(\ref{seq}). Let \( J \) be the largest subset of \( \{ 1, \dots, n \} \)
	such that \( \sum_{j \in J} c_j x_j = 0 \). Clearly \( J \) is a proper subset (since \( M \neq 0 \)),
	so \( I = \{ 1, \dots, n \} \setminus J \) is nonempty. By maximality of \( J \), it follows that \( \{ x_j \}_{j \in I} \)
	is a non-degenerate solution to the equation \( \sum_{j \in I} c_j x_j = M \). Choose any \( i \in I \).
	Then \( x_i \in V(I, i) \), and so \( x_i \in V_i \) as required.
\end{proof}

There is a connection between T-sequences and integer equations.
\begin{lemma}\label{newcrit}
	A sequence of integers \( \{ s_n \}_{n \in \mathbb{N}} \) is a T-sequence if and only if
	for all integers \( c_1, \dots, c_k \) and all nonzero integers \( M \) there exists
	a number \( m \) such that
	\[ c_1 s_{m_1} + \dots + c_k s_{m_k} \neq M \]
	for all \( m_1, \dots, m_k \) with \( m < m_1 < \dots < m_k \).
\end{lemma}
\begin{proof}
	It follows from \cite[Theorem~2]{prot}. See also \S~3 in \cite{prot}.
\end{proof}

The following lemma shows that a growing sequence of \( S \)-integers is a T-sequence.

\begin{lemma}\label{mainlem}
	Let \( S \) be a finite set of primes and let \( \{ s_n \}_{n \in \mathbb{N}} \)
	be a sequence of \( S \)-integers. If \( \lim_{n \to \infty} |s_n| = \infty \),
	then \( \{ s_n \}_{n \in \mathbb{N}} \) is a T-sequence.
\end{lemma}
\begin{proof}
	Suppose that \( \{ s_n \}_{n \in \mathbb{N}} \)	is not a T-sequence.
	By Lemma~\ref{newcrit}, there exist integers \( c_1, \dots, c_k \) (which we may assume to be nonzero)
	and a nonzero integer \( M \) such that for all \( m \) we can ensure that the equality
	\[ c_1 s_{m_1} + \dots + c_k s_{m_k} = M \]
	holds for some \( m_1, \dots, m_k \) with \( m < m_1 < \dots < m_k \).
	From now on we assume that given a number \( m \), respective numbers \( m_1, \dots, m_k \) are
	also given. The choice of \( m_i \) depends on \( m \), but we drop that indication.

	Consider the respective \( S \)-unit equation:
	\[ \sum_{i = 1}^k c_i x_i = M, \]
	and recall that by Lemma~\ref{inhom}, there exist finite sets \( V_i \), \( i = 1, \dots, k \),
	such that for any solution \( (x_1, \dots, x_k) \in U_S^k \) of that equation there exists some \( i \) with
	\( x_i \in V_i \). Now, pick \( m \) large enough to make \( s_{m_i} \) greater by absolute value than any
	element from \( V_i \) for all \( i = 1, \dots, k \) (this is possible since \( \lim_{n \to \infty} |s_n| = \infty \)).
	But \( (s_{m_1}, \dots, s_{m_k}) \) is an \( S \)-unit solution to our equation, so we have \( s_{m_i} \in V_i \)
	for some \( i \). This is a contradiction.
\end{proof}

We proceed with the proof of the main result.

Let \( a_n \) be the \( n \)th positive \( S \)-integer (in the increasing order).
Define \( s_n = (-1)^n a_{\lfloor \frac{n+1}{2} \rfloor} \). Notice that \( \{ s_n \}_{n \in \mathbb{N}} \)
contains all nonzero \( S \)-integers and \( |s_n| \leq |s_{n+1}| \) for all \( n \).
Clearly \( \lim_{n \to \infty} |s_n| = \infty \), so \( \{ s_n \}_{n \in \mathbb{N}} \) is a T-sequence by Lemma~\ref{mainlem}.

Let \( \{ d_n \}_{n \in \mathbb{N}} \) be a sequence of \( S \)-integers with \( \lim_{n \to \infty} |d_n| = \infty \).
Let \( W \) be a neighborhood of \( 0 \) in \( \mathbb{Z}\{ s_n \} \). Since \( \{ s_n \}_{n \in \mathbb{N}} \)
converges to \( 0 \) in that topology, for some index \( m \) all numbers \( s_n \), \( n \geq m \), lie in~\( W \).
As \( \lim_{n \to \infty} |d_n| = \infty \), for some index \( k \) all numbers \( d_n \), \( n \geq k \),
lie in the set \( \{ s_n \mid n \geq m \} \) and, consequentially, lie in \( W \). The choice of \( W \) was arbitrary,
therefore the sequence \( \{ d_n \}_{n \in \mathbb{N}} \) converges to \( 0 \) in \( \mathbb{Z}\{ s_n \} \) as claimed.

\bigskip

\noindent
{\sl Saveliy V. Skresanov\\
Sobolev Institute of Mathematics, 4 Acad. Koptyug avenue\\
and\\
Novosibirsk State University, 1 Pirogova street,\\
630090 Novosibirsk, Russia\\
e-mail: skresan@math.nsc.ru}


\begin{thebibliography}{9}
	\bibitem{prot}
		E.G.~Zelenuk, I.V.~Protasov, 
		\emph{Topologies on Abelian groups},
		Math. USSR Izvestia. \textbf{37} (1991), 445--460.

	\bibitem{openpr}
		I.V.~Protasov, E.G.~Zelenuk,
		\emph{Topologies on \(\mathbb{Z}\) determined by sequences: seven open problems},
		Matematychni Studii. V.12, No.1., 1999.

	\bibitem{kourovka}
		\emph{Unsolved problems in group theory. The Kourovka Notebook},
		19th ed., Inst. of Mathematics, SO RAN, Novosibirsk,
		2018.

	\bibitem{gaps}
		J.~Maynard, \emph{Small Gaps between Primes},
		Annals of Mathematics, SECOND SERIES, 181, no. 1 (2015), 383--413.
	
	\bibitem{masasi}
		M.~Higasikawa, \emph{Non-productive duality properties of topological groups},
		Top. Proc.~25 Summer (2000), 207--216.

	\bibitem{sunits}
		J.-H.~Evertse, K.~Gy\H{o}ry,
		\emph{Unit Equations in Diophantine Number Theory},
		Cambridge Studies in Advanced Mathematics. Cambridge: Cambridge University Press (2015). 
\end{thebibliography}
\end{document}